\documentclass[12pt]{amsart}

\usepackage {amsmath, amsopn}
\usepackage[applemac]{inputenc}
\usepackage {amsfonts, amssymb}
\usepackage {latexsym}

\newtheorem {theorem} {Theorem} [section]
\newtheorem {lemma} {Lemma} [section]
\newtheorem {prop} {Proposition} [section]

\newtheorem{preremark}{Remark}[section]
  {\begin{preremark}\rm}{\end{preremark}}
   \newtheorem{preremark1}{Example}[section]

     \newtheorem{preremark2}{Definition}[section]
\newenvironment{defn}
  {\begin{preremark2}\rm}{\end{preremark2}}

\DeclareMathOperator{\divo}{div}

\begin{document}



\title{On a coupled PDE model for image restoration}

\author{V.B. Surya Prasath}


\author{Dmitry Vorotnikov}


\thanks{Research partially supported by CMUC and FCT (Portugal), through
European program COMPETE/FEDER.}




\begin{abstract} In this paper, we consider a new coupled PDE
model for image restoration. Both
the image and the edge variables are incorporated by coupling them into two
different PDEs. It is shown that the initial-boundary value problem has global in time dissipative solutions (in a sense going back to P.-L. Lions), and several properties of these solutions are established. This is a rough draft, and the final version of the paper will contain a modelling part and numerical experiments. 
\end{abstract}

\maketitle

\section {Introduction}

\newcommand{\sgn}{\text{sgn}}
\newcommand {\R} {\mathbb{R}}
\newcommand {\E} {\mathbf{E}}
\def\be{\begin{equation}}
\def\ee{\end{equation}}
\def\fr#1#2{\frac{\partial #1}{\partial #2}}

We consider the following problem

\begin{equation} \label{e1} \frac {\partial u(t,x)} {\partial t} = \divo (g(v(t,x))\nabla u(t,x)), \end{equation}
\begin{equation}  \label{e2} \frac {\partial v(t,x)} {\partial t}-\lambda(x)\Delta v(t,x)= (1-\lambda(x))(|\nabla u(t,x)|-v(t,x)), \end{equation}
\be \label{e3} u \Big | _ {\partial\Omega} =0, v \Big | _ {\partial\Omega} =0,\end{equation}
\begin{equation} \label{e6}  u | _ {t=0} = u_0, \ v | _ {t=0} = v_0.\end{equation} 

It is shown that the problem possesses global in time \emph{dissipative solutions}; uniqueness, regularity and some other properties of these solutions are studied.  The concept of dissipative solution was suggested in \cite{blions} for the Euler equations of ideal fluid flow, which are still not proven to have global weak solvability. Later, existence of dissipative solutions was established for Boltzmann's  equation \cite{lio1,vilg}, the ideal MHD equations \cite{anwu}, Navier-Stokes-Maxwell equations \cite{sra}, Euler-$\alpha$ and Maxwell-$\alpha$ models \cite{NON11} and viscoelastic diffusion equations \cite{diss1}.  

The features of our problem \eqref{e1}--\eqref{e6} which oppose strong and classical weak well-posedness are the presence of a nonlinear function (modulus) of the gradient of $u$ in the right-hand side of \eqref{e2} and the Perona-Malik-like form of $g$. The inequality \eqref{dss} in the definition of dissipative solutions turns out to contain the absolute value function as well. Therefore, unlike in the previous works on dissipative solutions,  it is impossible to pass to the limit in this inequality via weak and weak-* compactness argument. Nevertheless, we manage to do it via strong compactness, although it is not sufficiently strong to obtain classical (i.e. not dissipative) weak solutions. 

\section {Well-posedness of the problem}

The objective of this section is to prove Theorem \ref{mais} concerning existence, uniqueness, regularity and some other properties of dissipative solutions to the problem \eqref{e1}--\eqref{e6}. We consider the simplest Dirichlet boundary condition \eqref{e3}, but other boundary conditions can also be handled. 

In the section, $\Omega$ is considered to be a domain (i.e. an open set in $\R^2$) possessing the \emph{cone property}. We recall \cite{adams} that this means that each point $x\in\Omega$ is a vertex of a \emph{finite cone} $C_x$ contained in $\Omega$, and all these cones $C_x$ are congruent. A finite cone is a set of the form $$C_x=B_1\cap \{x+\xi (y-x)|y\in B_2, \xi >0 \}$$ where $B_1$ and $B_2 $ are open balls in $\R^2$, $B_1$ is centered at $x$, and $B_2$ does not contain $x$.
Obviously, rectangular domains have this property. 

The symbol $C$ will stand for a generic positive constant that can take different values in different lines. We sometimes write $C(\dots)$ to specify that the constant depends on a certain parameter or value.

We assume that $g:\R \to \R$, $\frac 1 {\sqrt g}$ and $\lambda:\Omega \to \R$ are Lipschitz functions having positive values, $g$ is bounded, $\lambda\leq 1$, \be\lambda_0=\inf_{x\in\Omega}\lambda(x)>0.\ee

The assumptions on $g$ hold, for instance, if \be g(s)=\frac a {b+c|s|^d},\ee where $a,b,c,d$ are positive numbers, and $1\leq d\leq 2$. 

Note that \be\label{estg}\frac 1 {\sqrt {g(s)}}\leq \Big| \frac 1 {\sqrt {g(s)}}- \frac 1 {\sqrt {g(0)}} \Big| + \frac 1 {\sqrt {g(0)}}\leq C(g)(1+|s|).\ee

We use the standard notations $L_p (\Omega) $, $W_p^{m}
(\Omega) $, $H^{m} (\Omega) =$ $W_2^{m} (\Omega) $ for the Lebesgue
and Sobolev spaces. We will often keep the function space symbol and omit $\Omega$.

The Euclidean norm in finite-dimensional spaces is denoted by $ | \cdot | $. The symbol $ \| \cdot \| $ will stand for the Euclidean norm in $L_2(\Omega)$. The corresponding scalar products is denoted by a dot $\cdot$ and parentheses $(\cdot,\cdot)$. 

Let $H_0^1(\Omega)$ be the closure of the set of smooth, compactly supported in $\Omega$, functions in $H^1(\Omega)$. By virtue of Friedrichs' inequality, the Euclidean norm $\|\cdot\|_1$ corresponding to the scalar product
$$(u,v)_1=(\nabla u,\nabla v)$$ is a norm in $H_0^1$. 

The set $V_2=H_0^1(\Omega)\cap H^2(\Omega)$ is a Hilbert space with the scalar product
$$(u,v)_2=(u,v)_1+\sum_{|\alpha|=2}(D^\alpha u, D^\alpha v).$$ Denote the corresponding Euclidean norm by $\|\cdot\|_2$.

Let $V_r$, $1<r<2$, be the closure of $V_2$ in $W^1_r$.  
 
We recall the following abstract observation \cite{temam, book}. Assume that we have two Hilbert spaces, $X\subset Y,$ with
continuous embedding operator $i:X\to Y$, and $i(X)$ is dense in
$Y$. The adjoint operator $i^*:Y^*\to X^*$ is continuous and,
since $i(X)$ is dense in $Y$, one-to-one. Since $i$ is one-to-one,
$i^*(Y^*)$ is dense in $X^*$, and one may identify $Y^*$ with a
dense subspace of $X^*$. Due to the Riesz representation theorem,
one may also identify $Y$ with $Y^*$. We arrive at the chain of
inclusions:
\begin{equation} X\subset Y\equiv Y^*\subset X^*.\end{equation}
Both embeddings here are dense and continuous. Observe that in
this situation, for $f\in Y, u\in X$, their scalar product in $Y$
coincides with the value of the functional $f$ from $X^*$ on the
element $u\in X$:
\begin{equation} \label{ttt} (f,u)_Y=\langle f,u \rangle. \end{equation}
Such triples $(X,Y, X^*)$ are called Lions triples. 
We use the Lions triples $(V_2,L_2,V_2^*)$ and $(H^1_0,L_2,H^{-1})$. 

The symbols $C (\mathcal{J}; E) $, $C_w (\mathcal{J}; E) $, $L_2
(\mathcal{J}; E) $ etc. denote the spaces of continu\-ous, weakly
continuous, quadratically integrable etc. functions on an interval
$\mathcal{J}\subset \mathbb {R} $ with
values in a Banach space $E $. We recall that a function $u:
\mathcal{J} \rightarrow E$ is \textit{weakly continuous} if for
any linear continuous functional $g$ on $E$ the function  $g(
u(\cdot)): \mathcal{J}\to \mathbb{R}$ is continuous.

We require the following spaces
$$ W_1= W_1(\Omega, T) = \{\tau\in L_2 (0, T; V_2), \
\tau' \in L_2 (0, T;
 {V_2^{*}}) \},$$ $$\|\tau\|_{W_1}=\|\tau\|_{L_2 (0, T; V_2)}+\|\tau'\|_{L_2 (0, T;
 {V_2^{*}}) },$$
 $$ W_2= W_2(\Omega, T) = \{\tau\in L_2 (0, T; H^1_0), \
\tau' \in L_2 (0, T;
 H ^ {-1}) \},$$ $$\|\tau\|_{W_2}=\|\tau\|_{L_2 (0, T; H^1_0)}+\|\tau'\|_{L_2 (0, T; H ^ {-1})}.$$ 
 

Let us introduce the operator
 $$A:V_2\to V^{*}_2,\  \langle A u , \varphi \rangle= (u,\varphi)_2,$$ where $\varphi$ is an arbitrary element of $V_2$. 
 
Denote by $\mathcal{R}$ the following class of pairs of functions: 
$$\mathcal{R}=L_{4,loc} (0, \infty; V_2)\cap L_\infty(0,\infty; W^1_\infty)\cap W^1_{4,loc}(0,\infty;L_2)$$ $$\times  L_{2,loc} (0, \infty; V_2)\cap L_\infty(0,\infty;L_\infty)\cap W^1_{2,loc}(0,\infty;L_2).$$

Observe that the following expressions, where $\delta$ is a positive number, are well-defined for $(w,\tau)\in \mathcal{R}$, and their values are in $L_{2,loc}(0,\infty;L_2)$: 
 \be E_1(w,\tau,\delta)=-\frac {\partial w} {\partial t} + \delta\divo (g(\tau)\nabla w),\ee 
 \be E_2(w,\tau,\delta)= -\frac {\partial \tau} {\partial t}+\lambda\Delta \tau+ \delta(1-\lambda)(|\nabla w|-\tau)+(1-\delta)(\nabla \tau\cdot \nabla \lambda), \ee
  \be E_1(w,\tau)=E_1(w,\tau,1),\ee 
 \be E_2(w,\tau)=E_2(w,\tau,1). \ee

Let us recall the Sobolev inequality
  \begin {equation} \label{inq1} \|u\|_{L_\infty}\leq C(\Omega) \| u \| _2 ,\ u\in
V_2,
\end {equation}
and the Ladyzhenskaya inequality \cite{book}
\be\label{lady} \|u^2\|\leq \sqrt{2} \| u \|\, \| \nabla u \|,\ u\in
H^1_0.\ee

The following Gronwall-like lemma will be useful.
\begin{lemma} \label{ineq} (\cite[Lemma 3.1]{diss1}) Let $f,\chi, L, M: [0,T]\to \R$ be scalar functions, $\chi, L,M \in L_1(0,T),$ and $f\in
W^1_1(0,T)$ (i.e. $f$ is absolutely continuous). If \be \chi(t)
\geq 0, L(t) \geq 0\ee  and \be f'(t)+ \chi(t)\leq L(t)f(t)+M(t)
\ee for almost all $t\in (0,T)$, then \be f(t)+ \int\limits_0^t
\chi(s)\, ds\leq $$ $$ \exp\left(\int\limits_0^t L(s) ds\right)\left[f(0)+
\int\limits_0^t \exp\left(\int\limits_s^0 L(\xi) d\xi\right) M(s)\, ds
\right]\ee for all $t\in [0,T]$.
\end{lemma}

We can now give


\begin{defn} \label{maindef} Let $u_0,v_0\in L_2(\Omega)$. A pair of
functions $(u,v)$ from the class \be u,v\in C_w([0,\infty); L_2),\ee is called a {\it dissipative} solution
to problem \eqref{e1} -- \eqref{e6} if, 
for all test 
functions
$(\zeta,\theta)\in \mathcal{R}$ and all non-negative
moments of time $t$,
one has \be \label{dss} \gamma^{\|u(t)\|^2}\left[\|u(t)-\zeta(t)\|^2 +
\|v(t)-\theta(t)\|^2\right] $$

$$\leq
\gamma^{2t+\|u_0\|^2}\Big\{ 
\|u_0 -\zeta(0)\|^2 +
\|v_0-\theta(0)\|^2 $$ $$+
\int\limits_0^t 2\gamma^{-s}
\Big|\big( E_{1}(\zeta,\theta)(s),u(s)-\zeta(s)\big)+ \big( E_{2}(\zeta,\theta)(s), v(s)-\theta(s)\big)\Big| \Big\}
\ee where $\gamma$= $\gamma(\Omega, g, \lambda, \zeta,\theta)> 1$ is a certain function of $\Omega, g, \lambda, \zeta$ and $\theta$. \end{defn}
 

\begin{theorem} \label{mais} a) Given $u_0, v_0\in L_2$, there is a
dissipative solution to problem \eqref{e1} -- \eqref{e6}.

b) This solution $(u,v)$ belongs to $L_{4/3,loc}(0,\infty;V_{-\epsilon+4/3})\times L_{2,loc}(0,\infty; H^1_0)$, $\ 0<\epsilon<\frac 13$. 

c) If, for some $u_0,v_0\in L_2$, there exist $T>0$
and a strong solution $(u_T,v_T)$ to problem \eqref{e1} -- \eqref{e6}, which is a restriction of a function from $\mathcal{R}$ to $(0,T)$. Then the restriction
of any dissipative solution (with the same initial data) to
$(0,T)$ coincides with $(u_T,v_T)$.

d) Every strong solution $(u,v)\in \mathcal{R}$ is a (unique) dissipative solution.

e) The dissipative solutions satisfy the initial condition \eqref{e6}.

\end{theorem}

To prove Theorem \ref{mais}, we consider the following auxiliary problem:

\begin{equation} \label{au1s} \frac {\partial u} {\partial t}+ \varepsilon  A u = \delta\divo (g(v)\nabla u), \end{equation}
\begin{equation}  \label{au2s} \frac {\partial v} {\partial t}-\lambda\Delta v= \delta(1-\lambda)(|\nabla u|-v)+(1-\delta)(\nabla v\cdot \nabla \lambda), \end{equation}
\be \label{au3s} u \Big | _ {\partial\Omega} =0, v \Big | _ {\partial\Omega} =0,\end{equation}
\begin{equation} \label{au6s}  u | _ {t=0} =\delta u_0, \ v | _ {t=0} =\delta v_0.\end{equation} 

Here, $\varepsilon>0$ and $0\leq \delta \leq 1$ are parameters.

The weak formulation of \eqref{au1s} -- \eqref{au6s} is as follows. 

\begin{defn} A pair of
functions $(u,v)$ from the class \be u \in W_1,\ v\in
W_2\ee is a {\it weak} solution to problem \eqref{au1s} -- \eqref{au6s} if the equalities

   \be\label{sw1}\frac {d} {d t} (u, \varphi)   +  \varepsilon (u,\varphi)_2 + \delta(g(v)\nabla u, \nabla \varphi) = 0,
   \end{equation}
   and
    \be\label{sw2}\frac {d} {d t} (v, \phi) +  ({\lambda}\nabla v, \nabla \phi)$$ $$+ \delta(\nabla v, \phi \nabla {\lambda}) - \delta\left((1-\lambda)(|\nabla u|-v), \phi\right)=0 \end{equation}   are satisfied  for all
 $ \varphi\in V_2, \ \phi\in H^1_0$ almost everywhere in $(0, T)$, and \eqref{au3s} and \eqref{au6s} hold. \end{defn}

 \begin{lemma} \label{diins} Let $(u,v)$ be a weak solution to problem \eqref{au1s} -- \eqref{au6s}. Then, for all 
 test functions $(\zeta,\theta)\in \mathcal{R}$ and $0\leq t\leq T$,
one has \be \label{disdels} \gamma^{\|u(t)\|^2}\big\{\|u(t)-\zeta(t)\|^2 +
\|v(t)-\theta(t)\|^2 
$$ $$+ 2\varepsilon\int\limits_0^t
\|u(s)-\zeta(s)\|_2^2\, ds+\lambda_0\int\limits_0^t\|v(s)-\theta(s)\|_1^2\, ds\big\}$$ $$\leq
\gamma^{2t+\delta\|u_0\|^2}\Big\{ 
\|\delta u_0 -\zeta(0)\|^2 +
\|\delta v_0-\theta(0)\|^2 $$ $$+
\int\limits_0^t 2\gamma^{-s}
\Big|\big( E_{1}(\zeta,\theta,\delta)(s),u(s)-\zeta(s)\big)$$
$$+  \big( E_{2}(\zeta,\theta,\delta)(s), v(s)-\theta(s)\big)-\varepsilon (\zeta(s),u(s)-\zeta(s))_2\Big| \, ds \Big\}
\ee where $\gamma$= $\gamma(\Omega, g, \lambda, \zeta,\theta)> 1$ is a certain function of $\Omega, g, \lambda, \zeta$ and $\theta$. \end{lemma}

\begin{proof} Let us first derive the straightforward energy estimate. For almost all $t\in(0,T)$, let $\varphi=u(t)$ in \eqref{sw1}.  Then\footnote{See e.g. \cite[p. 153]{book} on how $\frac 12$ appears in \eqref{ee1}.} 
\be\label{ee1} \frac 12 \frac {d} {d t} (u, u) +  \delta
   (g(v)\nabla u, \nabla u) +\varepsilon (u,u)_2 = 0.
   \end{equation} Integration in time gives     
   \be\label{ee2} \frac 12 \|u(t)\|^2 + \int\limits_0^t  
   (\delta g(v(s))\nabla u(s), \nabla u(s))\,ds \leq  \frac \delta 2 \|u_0\|^2.
   \end{equation}  

Observe now that \be\label{sz1} \frac {d} {d t} (\zeta, \varphi) +  \delta
   (g(\theta)\nabla \zeta, \nabla \varphi) +(E_{1}(\zeta,\theta,\delta),\varphi)+\varepsilon (\zeta,\varphi)_2 = \varepsilon (\zeta,\varphi)_2,
   \end{equation}
   and
    \be \label{sz2} \frac {d} {d t} (\theta, \phi) +  ({\lambda} \nabla \theta, \nabla \phi)+\delta ( \nabla \theta, \phi \nabla {\lambda})$$ $$- \delta\left((1-\lambda)(|\nabla \zeta|-\theta), \phi\right) + (E_{2}(\zeta,\theta,\delta),\phi)=0.\end{equation} 
      for
 $ \varphi\in V_2, \ \phi\in H^1_0$. 
Denote $w=u-\zeta$ and $\varsigma=v-\theta$. For almost all $t\in(0,T)$, put $\varphi=w(t)$ and $\phi=\varsigma(t)$. Add the difference between \eqref{sw1} and \eqref{sz1} with the difference between \eqref{sw2} and \eqref{sz2}, arriving at

\be \label{ins1} \frac 12 \frac {d} {d t} (w, w) +\frac 1 2 \frac {d} {d t} (\varsigma, \varsigma)+\delta (g(v)\nabla w, \nabla w)  $$ $$+\varepsilon (w,w)_2+({\lambda} \nabla \varsigma, \nabla \varsigma)+\delta\left((1-\lambda) \varsigma, \varsigma\right)$$ $$=-\delta ([g(v)-g(\theta)]\nabla \zeta, \nabla w)+\delta\left((1-\lambda)(|\nabla u|-|\nabla \zeta|), \varsigma\right)-\delta( \nabla \varsigma, \varsigma \nabla {\lambda})$$ $$+(E_{1}(\zeta,\theta,\delta),w)+ (E_{2}(\zeta,\theta,\delta),\varsigma)-\varepsilon (\zeta,w)_2.\ee

Let us estimate the first three terms in the right-hand side. 
\be-\delta ([g(v)-g(\theta)]\nabla \zeta, \nabla w)+\delta\left((1-\lambda)(|\nabla u|-|\nabla \zeta|), \varsigma\right)$$ $$\leq C(\zeta,g)\delta (|v-\theta|, |\nabla w|) $$ $$\leq C(\zeta,g)\left(\frac {|\varsigma|}{\sqrt{g(v)}}, \sqrt{\delta g(v)}|\nabla w|\right)$$ $$=C(\zeta,g)\left[\left(\frac {|\varsigma|}{\sqrt{g(0)}}, \sqrt{\delta g(v)}|\nabla w|\right)+\left(|\varsigma|\left(\frac {1}{\sqrt{g(\theta)}}-\frac {1}{\sqrt{g(0)}}\right), \sqrt{\delta g(v)}|\nabla w|\right)\right]$$ $$+ C(\zeta,g)\left(|\varsigma|\left(\frac {1}{\sqrt{g(v)}}-\frac {1}{\sqrt{g(\theta)}}\right), \sqrt{\delta g(v)}|\nabla w|\right)$$ $$\leq C(\zeta,\theta,g)\left(|\varsigma|, \sqrt{\delta g(v)}|\nabla w|\right)+C(\zeta,g)\left(\varsigma^2, \sqrt{\delta g(v)}(|\nabla \zeta|+|\nabla u|)\right)$$ $$\leq \|\sqrt{\delta g(v)}\nabla w\|^2+ C(\zeta,\theta,g)\|\varsigma\|^2+C(\zeta,g)\left(\varsigma^2, \sqrt{\delta g(v)}|\nabla u|\right),\ee
and \be-\delta( \nabla \varsigma, \varsigma \nabla {\lambda})\leq C(\lambda) (\varsigma, \nabla \varsigma)\leq \frac {\lambda_0} 4 \|\varsigma\|^2_1+C(\lambda) \|\varsigma\|^2 \ee

Now, \eqref{ins1} implies 

\be \label{ins2} \frac 12 \frac {d} {d t} (w, w) +\frac 1 2 \frac {d} {d t} (\varsigma, \varsigma) +\varepsilon (w,w)_2+\frac {3 \lambda_0} 4 \|\varsigma\|^2_1 $$ $$\leq  C(\zeta,\theta,\lambda,g)\left(\varsigma^2, 1+\sqrt{\delta g(v)}|\nabla u|\right)$$ $$+(E_{1}(\zeta,\theta,\delta),w)+ (E_{2}(\zeta,\theta,\delta),\varsigma)-\varepsilon (\zeta,w)_2.\ee
Denote $\Phi(t)=\left\|1+\sqrt{\delta g(v(t))}|\nabla u(t)|\right\|$. Due to \eqref{lady}, 
\be \label{ins2} \frac {d} {d t} (w, w) +\frac {d} {d t} (\varsigma, \varsigma) +2\varepsilon (w,w)_2+\frac {3 \lambda_0} 2 \|\nabla\varsigma\|^2 $$ $$\leq  C(\zeta,\theta,\lambda,g)\Phi \|\varsigma\|\|\nabla\varsigma\| $$ $$+2(E_{1}(\zeta,\theta,\delta),w)+ 2(E_{2}(\zeta,\theta,\delta),\varsigma)-2\varepsilon (\zeta,w)_2. \ee
Thus,
\be \label{ins3} \frac {d} {d t} \|w\|^2 +\frac {d} {d t} \|\varsigma\|^2 +2\varepsilon \|w\|^2_2+\lambda_0 \|\nabla\varsigma\|^2 $$ $$\leq  C(\zeta,\theta,\lambda,g)\Phi^2 \|\varsigma\|^2 $$ $$+2(E_{1}(\zeta,\theta,\delta),w)+ 2(E_{2}(\zeta,\theta,\delta),\varsigma)-2\varepsilon (\zeta,w)_2. \ee

We now require two estimates for $\Phi$,
\be\label{phi1}\int\limits_0^t  \Phi^2(s) ds= \int\limits_0^t  \int\limits_\Omega^{} [1+\sqrt{\delta g(v(s))}|\nabla u(s)|]^2\,dx \, ds$$ $$\leq 2 \int\limits_0^t  \int\limits_\Omega^{}\,dx \, ds+ 2\int\limits_0^t\int\limits_\Omega^{} \delta g(v(s))|\nabla u(s)|^2\,dx \, ds$$ $$\leq 2t|\Omega|+\delta \|u_0\|^2-\|u(t)\|^2, \ee
by virtue of \eqref{ee2}, and
\be\label{phi2}\int\limits_0^t  \Phi^2(s) ds\geq \int\limits_0^t  \int\limits_\Omega^{}\,dx \, ds=t|\Omega|. \ee

With the help of Lemma \ref{ineq}, we derive from \eqref{ins3}-- \eqref{phi2} that 
\be \label{bine}\|w(t)\|^2 +\|\varsigma(t)\|^2+ 2\varepsilon \int\limits_0^t
\|w(s)\|^2_2 \, ds + \lambda_0 \int\limits_0^t  \|\nabla\varsigma(s)\|^2\, ds $$ $$\leq \exp\left(C(\zeta,\theta,\lambda,g)\int\limits_0^t  \Phi^2(s) ds\right)\Big\{\|w(0)\|^2 +\|\varsigma(0)\|^2+ $$ $$
\int\limits_0^t \exp\left(C(\zeta,\theta,\lambda,g)\int\limits_s^0  \Phi^2(\xi) d\xi\right)[ 2(E_{1}(\zeta,\theta,\delta)(s),w(s))$$ $$ + 2(E_{2}(\zeta,\theta,\delta)(s),\varsigma(s))-2\varepsilon (\zeta(s),w(s))_2]\, ds
\Big\}$$ $$
\leq \exp\left(C(\zeta,\theta,\lambda,g)(2t|\Omega|+\delta \|u_0\|^2-\|u(t)\|^2)\right)\Big\{\|w(0)\|^2 +\|\varsigma(0)\|^2+ $$ $$
\int\limits_0^t \exp\left(-C(\zeta,\theta,\lambda,g)s|\Omega|\right)\big|2(E_{1}(\zeta,\theta,\delta)(s),w(s))$$ $$+ 2(E_{2}(\zeta,\theta,\delta)(s),\varsigma(s))-2\varepsilon (\zeta(s),w(s))_2\big|\, ds
\Big\}$$
$$
\leq \exp\left(C(\zeta,\theta,\lambda,g)(|\Omega|+1)(2t+\delta \|u_0\|^2-\|u(t)\|^2)\right)\Big\{\|w(0)\|^2 +\|\varsigma(0)\|^2+ $$  $$
\int\limits_0^t \exp\left(-C(\zeta,\theta,\lambda,g)s(|\Omega|+1)\right)\big| 2(E_{1}(\zeta,\theta,\delta)(s),w(s))$$ $$+ 2(E_{2}(\zeta,\theta,\delta)(s),\varsigma(s))-2\varepsilon (\zeta(s),w(s))_2\big|\, ds
\Big\},\ee since $s\leq 2t$. Now \eqref{bine} yields \eqref{disdels} with $$\gamma=\exp\{C(\zeta,\theta,\lambda,g)(|\Omega|+1)\}.$$
\end{proof}

\begin{lemma} \label{leaes} Let $(u,v)$ be a weak solution to problem \eqref{au1s} -- \eqref{au6s}. The following estimates are valid: \be \label{ae1s}
\|u\|_{L_\infty(0,T;L_2)}+\|v\|_{L_\infty(0,T;L_2)}+\|v\|_{L_2(0,T;H^1_0)}\leq C,\ee 
\be \label{ae2s}
\|u\|_{L_2(0,T;V_2)}\leq  \frac {C}{\sqrt{\varepsilon}},\ee
\be \label{ae5s}
\|\nabla u\|_{L_2(0,T;L_1)}+\|\nabla u\|_{L_1(0,T;L_r)}+\|\nabla u\|_{L_{4/3}(0,T;L_{-\epsilon+4/3})}\leq C,\ee $$\ 1<r<2, 0<\epsilon<\frac 13,$$
\be\label{ae3s} \|u'\|_{L_2(0,T;V_2^*)}+\|v'\|_{L_2(0,T;H^{-2})}\leq (1+
\sqrt{\varepsilon})C,\ee \be\label{ae4s}
\|v'\|_{L_2(0,T;H^{-1})}\leq  (1+1/\sqrt{\varepsilon})C.\ee   The constants $C=C(T,\|u_0\|,\|v_0\|,\lambda,g,\Omega)$ are
independent of $\varepsilon$ and $\delta$.
\end{lemma}

\begin{proof} The estimates \eqref{ae1s} and \eqref{ae2s} are direct consequences of \eqref{disdels} with $\zeta \equiv \theta\equiv 0$. 

Then, using \eqref{estg} and \eqref{ee2}, we have \be\|\nabla u\|_{L_2(0,T;L_1)}\leq $$  $$ \|\sqrt{\delta g(v)}\nabla u\|_{L_2(0,T;L_2)}\|1/\sqrt{g(v)}\|_{L_\infty(0,T;L_2)} $$ $$\leq C\|1+|v|\|_{L_\infty(0,T;L_2)}\leq C, \ee and, since $H^1_0\subset L_p$ for any $p<\infty$ by Sobolev embedding,
\be\|\nabla u\|_{L_1(0,T;L_r)}\leq $$  $$ \|\sqrt{\delta g(v)}\nabla u\|_{L_2(0,T;L_2)}\|1+|v|\|_{L_2(0,T;L_{2r/(2-r)})} $$ $$\leq C (1+\|v\|_{L_2(0,T;H^1_0)}) \leq C. \ee
By the time-space H\"{o}lder inequality \cite[Lemma 2.2.1(b)]{book}, \be\|\nabla u\|_{L_{4/3}(0,T;L_{-\epsilon+4/3})}\leq \||\nabla u|^{1/2}\|_{L_{4}(0,T;L_{2})} \||\nabla u|^{1/2}\|_{L_{2}(0,T;L_{\frac {8-6\epsilon}{2+3\epsilon}})}$$ $$\leq \sqrt{\|\nabla u\|_{L_{2}(0,T;L_{1})} \|\nabla u\|_{L_{1}(0,T;L_{\frac {4-3\epsilon}{2+3\epsilon}})}}\leq C.\ee

It
remains to estimate the time derivatives, expressing them from \eqref{sw1} and \eqref{sw2}. Utilizing \eqref{ee2}, we get
 \be\label{sdert}\|\langle u', \varphi \rangle\|_{L_2(0,T)}\leq \delta \|(g(v)\nabla u, \nabla \varphi)\|_{L_2(0,T)}
  +  \varepsilon \|(u,\varphi)_2\|_{L_2(0,T)}$$ $$\leq \|\sqrt{\delta g(v)}\|_{L_\infty(0,T;L_\infty)}\|\sqrt{\delta g(v)}\nabla u\|_{L_2(0,T;L_2)} \|\nabla \varphi\|+\sqrt\varepsilon \sqrt\varepsilon \|u\|_{L_2(0,T;V_2)}\|\varphi\|_2 $$ $$ \leq C(1+
\sqrt{\varepsilon})\|\varphi\|_2, \ee
and 
   \be\label{sderts}\|\langle v', \phi \rangle\|_{L_2(0,T)}\leq \|({\lambda}\nabla v, \nabla \phi)\|_{L_2(0,T)}+ \delta\|(\nabla v, \phi \nabla {\lambda})\|_{L_2(0,T)}$$ $$ +\delta\|\left((1-\lambda)v, \phi\right)\|_{L_2(0,T)}+ \delta\|\left((1-\lambda)|\nabla u|, \phi\right)\|_{L_2(0,T)}$$ $$\leq\|v\|_{L_2(0,T;H^1_0)}\|\phi\|_1+ C(\lambda)\|v\|_{L_2(0,T;H^1_0)}\|\phi\| $$ $$ +\|\nabla u\|_{L_2(0,T;L_1)}\|\phi\|_{L_\infty} \leq C\|\phi\|_2.\ee    
In order to get \eqref{ae4s}, it suffices to observe that \be\delta\|\left((1-\lambda)|\nabla u|, \phi\right)\|_{L_2(0,T)}$$ $$\leq \|\nabla u\|_{L_2(0,T;L_2)}\|\phi\|\leq C\|u\|_{L_2(0,T;V_2)}\|\phi\|_1\leq \frac C{\sqrt{\varepsilon}} \|\phi\|_1. \ee
 \end{proof}

\begin{lemma} \label{lews} Given $T>0$ and $u_0,v_0\in L_2$, there exists a weak solution to problem \eqref{au1s} -- \eqref{au6s} with $\delta=1$. \end{lemma}

\begin{proof}  Let us rewrite the weak statement of \eqref{au1s} -- \eqref{au6s} in the suitable operator form \begin{equation} \label{opeq} \tilde {A} (u, v) = \delta Q (u, v). \end{equation} The operators
$\tilde {A},Q:W_1\times W_2\to L_2 (0, T; V^{*}_2)\times L_2 (0, T; H^{-1})
\times L_2\times L_2$ are determined by the formulas
$$\langle\tilde{A} (u,v), (\varphi,\phi)\rangle$$ $$= \Big(\frac {d} {d t} (u, \varphi)   +  \varepsilon (u,\varphi)_2, \frac {d} {d t} (v, \phi) +  ({\lambda}\nabla v, \nabla \phi), u | _ {t=0}, v | _ {t=0}\Big),
$$

$$\langle Q(u,v), (\varphi,\phi)\rangle$$ $$=\Big(-(g(v)\nabla u, \nabla \varphi), -(\nabla v, \phi \nabla {\lambda}) + \left((1-\lambda)(|\nabla u|-v), \phi\right), u_0, v_0\Big).$$
Here $\varphi\in V_2$ and $\phi\in H^1_0$ are test functions. 

The operator $Q$ is continuous and compact. Here we only explain this claim for its first component, and for the others the proof is more straightforward. We observe first that the embedding $W_1\subset L_p (0, T; W^1_p)$ is compact for some $p>2$. This can be shown using \cite[Corollary 8]{sim}. The embedding $W_2 \subset L_2 (0, T; L_2)$ is compact by \cite[Corollary 4]{sim}. Let $(u_m,v_m)\rightharpoonup (u_0,v_0)$ be a weakly converging sequence in $W_1\times W_2$. Then $(u_m,v_m)$ is strongly converging in  $ L_p (0, T; W^1_p)\times L_2 (0, T; L_2)$. By Krasnoselskii's theorem \cite[Theorem 2.1]{krasn}, $g(v_m)\to g(v_0)$ in $L_q (0, T; L_q)$ for any $q<+\infty$. Thus, $g(v_m)\nabla u_m\to g(v_0)\nabla u_0$ in $L_2 (0, T; L_2)$, and the claim follows. 
 
The linear operator $\tilde{A}$ is continuous by \cite[Corollary 2.2.3]{book} and invertible by \cite[Lemma 3.1.3]{book}. Thus, \eqref{opeq} can be rewritten as \begin{equation} \label{opeq} (u, v) = \delta \tilde{A}^{-1}Q (u, v) \end{equation}
 in the space $W_1\times W_2$.
 
Lemma \ref{leaes} yields the a priori estimate \be \label{aes}
\|u\|_{W_1}+ \|v\|_{W_2}\leq C,\ee  where $C$ may depend on $\varepsilon$ but does not depend on $\delta$. By Schaeffer's theorem \cite[p. 539]{evans}, there exists a fixed point of the map $\tilde{A}^{-1}Q $, which is the required solution.  
\end{proof}

We will also need the following simple fact. 

\begin{prop} \label{prok} Let $G$ be a measurable set in a finite-dimensional space, $\chi:\R\to \R$ be a continuous function, and let $y_m:G\to \R$ be a sequence of functions. Assume that $\{y_m\}$ is uniformly bounded in $L_\infty(G)$, and $y_m\to y_0$ in $L_q(G)$, $q\geq 1$. Then $\chi(y_m)\to \chi(y_0)$ in  $L_p(G)$ for any $p<\infty$.  \end{prop}

\begin{proof} Due to the uniform boundedness of $\{y_m\}$, without loss of generality we may assume that $\chi$ is also bounded, and then it suffices to apply \cite[Theorem 2.1]{krasn}. \end{proof}

Based on the obtained lemmas, we can proceed with the sketch of the \textbf{proof of Theorem \ref{mais}}. We refer to \cite{NON11} for the details of the technique, and mainly focus on the new issues. To prove a) and b), one passes to the limit in \eqref{disdels} with $\delta=1$ as $\varepsilon=\varepsilon_m\to 0$ on every interval $(0,T)$, $T>0$. However, unlike in \cite{blions,anwu,diss1,NON11}, in view of the presence of the absolute value in the right-hand member of \eqref{disdels}, it is not possible to do it via weak and weak-* compactness.

Let $(u_m,v_m)$ be the weak solution to problem
\eqref{au1s} -- \eqref{au6s} with $\varepsilon=\varepsilon_m$. Lemma \ref{leaes}, \cite[Corollary 4]{sim} and the compact Sobolev embedding $W_{-\epsilon+4/3}^1\subset L_2$ imply that without loss of generality
$u_m\to u$ in $L_{4/3}(0,T;L_2)$, $v_m\to v$ in $L_2(0,T;L_2)$. Then, by \eqref{ae1s} and Proposition \ref{prok}, $$\gamma^{\|u_m(t)\|^2}\to\gamma^{\|u(t)\|^2}$$ in $L_2(0,T)$. Furthermore, by the same proposition, $\|u_m(t)-\zeta(t)\|^2 \to \|u(t)-\zeta(t)\|^2$, $\|v_m(t)-\theta(t)\|^2 \to\|v(t)-\theta(t)\|^2$ in $L_2(0,T)$. Therefore
\be\gamma^{\|u_m(t)\|^2}\big\{\|u_m(t)-\zeta(t)\|^2 +
\|v_m(t)-\theta(t)\|^2 
\big\}$$ $$\to\gamma^{\|u(t)\|^2}\big\{\|u(t)-\zeta(t)\|^2 +
\|v(t)-\theta(t)\|^2 
\big\}\ee in $L_1(0,T)$. Note that \be\theta\in L_4(0,T;H^1)\subset L_\infty(0,T;L_2)\cap L_2(0,T;H^2).\ee This yields $E_{1}(\zeta,\theta)$ $\in L_{4}(0,T;L_2)$. Remember that $E_{2}(\zeta,\theta)\in L_{2}(0,T;L_2)$. Thus, we can pass to the limit in the right-hand side of \eqref{disdels} as well; the last summand (the one with $\varepsilon$) goes to zero due to \eqref{ae2s}. 

To get c), one lets $\zeta=u_T$,
$\theta=v_T$ in \eqref{dss} for $t\in(0,T)$, and then the right-hand member of \eqref{dss} vanishes there. And e) is obtained by putting $t=0$ in \eqref{dss} and applying a density argument. Finally, d) is a consequence of a), e)
and c).


\bibliography{stu}

\begin{thebibliography}{10}

\bibitem{adams}
R.~A. Adams.
\newblock {\em Sobolev spaces}.
\newblock Academic Press [A subsidiary of Harcourt Brace Jovanovich,
  Publishers], New York-London, 1975.
\newblock Pure and Applied Mathematics, Vol. 65.

\bibitem{sra}
D.~{Arsenio} and L.~{Saint-Raymond}.
\newblock {Maxwell's equations and the Lorentz force}.
\newblock www.math.univ-toulouse.fr/berestycki2011/Talks/Saint Raymond.pdf,
  June 2011.

\bibitem{evans}
L.~C. Evans.
\newblock {\em Partial differential equations}, volume~19 of {\em Graduate
  Studies in Mathematics}.
\newblock American Mathematical Society, Providence, RI, second edition, 2010.

\bibitem{vilg}
I.~M. Gamba, V.~Panferov, and C.~Villani.
\newblock Upper {M}axwellian bounds for the spatially homogeneous {B}oltzmann
  equation.
\newblock {\em Arch. Ration. Mech. Anal.}, 194(1):253--282, 2009.

\bibitem{krasn}
M.~A. Krasnosel'skii.
\newblock {\em Topological methods in the theory of nonlinear integral
  equations}.
\newblock Translated by A. H. Armstrong; translation edited by J. Burlak. A
  Pergamon Press Book. The Macmillan Co., New York, 1964.

\bibitem{lio1}
P.-L. Lions.
\newblock Compactness in {B}oltzmann's equation via {F}ourier integral
  operators and applications. {I}, {II}.
\newblock {\em J. Math. Kyoto Univ.}, 34(2):391--427, 429--461, 1994.

\bibitem{blions}
P.-L. Lions.
\newblock {\em Mathematical topics in fluid mechanics. {V}ol. 1}, volume~3 of
  {\em Oxford Lecture Series in Mathematics and its Applications}.
\newblock The Clarendon Press Oxford University Press, New York, 1996.
\newblock Incompressible models, Oxford Science Publications.

\bibitem{sim}
J.~Simon.
\newblock Compact sets in the space {$L^p(0,T;B)$}.
\newblock {\em Ann. Mat. Pura Appl. (4)}, 146:65--96, 1987.

\bibitem{temam}
R.~Temam.
\newblock {\em Navier-{S}tokes equations}, volume~2 of {\em Studies in
  Mathematics and its Applications}.
\newblock North-Holland Publishing Co., Amsterdam, revised edition, 1979.
\newblock Theory and numerical analysis, With an appendix by F. Thomasset.

\bibitem{NON11}
D.~{Vorotnikov}.
\newblock {Global generalized solutions for Maxwell-alpha and Euler-alpha
  equations}.
\newblock {\em ArXiv e-prints, to appear in Nonlinearity}, Dec. 2010.

\bibitem{diss1}
D.~A. Vorotnikov.
\newblock Dissipative solutions for equations of viscoelastic diffusion in
  polymers.
\newblock {\em J. Math. Anal. Appl.}, 339(2):876--888, 2008.

\bibitem{anwu}
J.~Wu.
\newblock Analytic results related to magneto-hydrodynamic turbulence.
\newblock {\em Phys. D}, 136(3-4):353--372, 2000.

\bibitem{book}
V.~G. Zvyagin and D.~A. Vorotnikov.
\newblock {\em Topological approximation methods for evolutionary problems of
  nonlinear hydrodynamics}, volume~12 of {\em de Gruyter Series in Nonlinear
  Analysis and Applications}.
\newblock Walter de Gruyter \& Co., Berlin, 2008.

\end{thebibliography}

\bibliographystyle{abbrv}
\end{document}